\documentclass[a4paper]{article}

   \usepackage{amsmath,amssymb,amsthm,amscd}

   \newtheorem{Thm}{Theorem}[section]
   \newtheorem{Lem}[Thm]{Lemma}
   \newtheorem{Cor}[Thm]{Corollary}

   \newtheorem{Rem}[Thm]{Remark}

   \makeatletter
   \@addtoreset{equation}{section}
   \makeatother
   
\newcommand{\R}{\mathbb{R}}

\begin{document}

   \title{The Generalized Dock Problem}

   \pagestyle{headings}

   \author{C.-I. Martin}



  \date{}
  \maketitle
  \today

\begin{abstract}
\noindent
We show existence and uniqueness for a linearized water wave problem in a two dimensional domain $G$ with corner,
formed by two semi-axis $\Gamma_1$ and $\Gamma_2$ which intersect
under an angle $\alpha\in (0,\pi ]$. The existence and uniqueness of the solution is proved by considering an auxiliary mixed problem with Dirichlet and Neumann boundary conditions. The latter guarantees the existence of the Dirichlet to Neumann map. The water wave boundary value problem is then shown to be equivalent to an equation like $v_{tt}+g\Lambda v=P$ with initial conditions, where $t$ stands for time, $g$ is the gravitational constant, $P$ means pressure, and $\Lambda$ is the Dirichlet to Neumann map. We then prove that $\Lambda$ is a positive self-adjoint operator.
\end{abstract}

\noindent  {\bf Keywords: water waves; corner domain; Dirichlet to Neumann map.}\\
\noindent {\bf Mathematics Subject Classification: } \\
\noindent {\bf Primary: 35Q35}\\
\noindent {\bf Secondary: 76B15
}\\


\section{Introduction}
This paper considers the wave motion in water with a free surface
and subjected to gravitational and other forces. Namely, a
dock-problem like will be studied. We first give a brief account on
the general theory of surface waves and then continue with the
statement of the dock problem and previous work in this direction.
We now summarize the  fundamental mathematical basis for our later
endeavors by formulating a typical problem which arises in the
hydrodynamics of surface waves. One of the first papers in this field belongs to Lord Rayleigh \cite{Ray}. For a thorough treatment of the
theory of water waves one could consult the books of
Stoker \cite{Stok} and Lamb \cite{Lamb}. See also the paper \cite{Stok1}.\\As for more recent work on water waves we mention the papers
\cite{Cra2, Cra3}.\\
Let us consider the physical situation of an ocean beach. The water
is assumed to be initially at rest occupying the region defined by
the equation
$$-h(x_{1},x_{2})\leq y\leq 0,\,\,x^{s}(x_2,t)\leq x_1 <+\infty,\,-\infty< x_2<+\infty,$$
where $x^{s}(x_2,t)$ is the horizontal coordinate of the water line on the
shore. We assume that at time $t=0$ a disturbance is created on the
surface of the water, and one then wants to determine the subsequent
motion of the water, namely the form of the free surface
$\eta(x_1,x_2; t)$ and the velocity field components $u,v,w$ as
functions of the space variables $x_1, x_2, y$ and the time $t$. We
will also assume all flows to be incompressible and irrotational.
The incompressibility of the flow gives the law of mass conservation
\begin{equation}
\label{masscons} {\rm{div}}\, {\bf v}=u_{x_1}+v_{x_2}+w_y=0,
\end{equation}
where ${\bf v}=(u,v,w)$ denotes the velocity field. Since the flow
is assumed to be irrotational we have that
\begin{equation}
\label{irrot} {\rm{curl}}\, {\bf v}=(w_{x_2}-v_y, u_y-w_{x_1},
v_{x_1}-u_{x_2})=0.
\end{equation}
The fact that ${\rm{curl}}\, {\bf v}=0$ implies the existence of a
single-valued velocity potential $\mathcal{V} (x_1,x_2,y,t)$ in any simple
connected region, i.e.,
\begin{equation}
\label{velpot} {\bf v}=\nabla\mathcal{V}=(\mathcal{V}_{x_1},\mathcal{V}_{x_2},\mathcal{V}_y).
\end{equation}
Equations $\eqref{masscons}$ and $\eqref{velpot}$ give that the
velocity potential $\mathcal{V}$ satisfies the Laplace equation
$$\Delta\mathcal{V} =0.$$

From the irrotational character of the water flow $\eqref{irrot}$ we
obtain the Bernoulli law
\begin{equation}
\label{Bernoulli}
\mathcal{V}_t+\frac{1}{2}(u^{2}+v^{2}+w^{2})+\frac{p}{\rho}+gy=C(t),
\end{equation}
with $C(t)$ depending only on $t$, and not on the space variables.
\\
\\
{\bf Boundary conditions}\\
In the problem under consideration it is assumed that the fluid has
a boundary surface $S$ which has the property that any particle
which is once on the surface remains on it.\\
Assume that $S$ is given by an equation $\xi(x_1, x_2, y, t)=0$.
Differentiation with respect to $t$ gives that the condition
\begin{equation}
\label{surder} \frac{d\xi}{dt}=u\xi_{x_1}+v\xi_{x_2}+w\xi_y+\xi_t=0
\end{equation}
 holds on $S$.
Using relations $\eqref{velpot}$, $\eqref{surder}$ and the fact that
the vector $(\xi_{x_1},\xi_{x_2}, \xi_y)$ is normal to $S$ we obtain
that
\begin{equation}
\label{norderphi}
\frac{\partial\mathcal{V}}{\partial\nu}=-\frac{\xi_t}{\sqrt{\xi_{x_1}^{2}+\xi_{x_2}^{2}+\xi_y^{2}}},
\end{equation}
where $\frac{\partial}{\partial\nu}$ means differentiation in the
direction of the normal to $S$.\\
An important special case is when the boundary $S$ is independent of
the time $t$ (the bottom of the sea for eg.), situation which leads
to the boundary condition
\begin{equation}
\label{tindep} \frac{\partial\mathcal{V}}{\partial\nu}=0\,\,\rm{on}\,\, S
\end{equation}
Another important situation is when the boundary surface $S$ is
given by the equation
\begin{equation}
\label{freesurfaceeq} y=\eta(x_1, x_2, t),
\end{equation}
and the surface is not prescribed apriori. In this case we have
$\xi=y-\eta(x_1, x_2, t)=0$ for any particle, and $\eqref{surder}$
leads to
\begin{equation}
\label{kincond} \mathcal{V}_{x_1}\eta_{x_1}-\mathcal{V}_y+
\mathcal{V}_{x_2}\eta_{x_2}+\eta_t=0\,\,\rm{on}\,\,y=\eta(x_1,x_2,t),
\end{equation}
while the Bernoulli's law gives the condition
\begin{equation}
\label{dyncond} g\eta
+\mathcal{V}_t+\frac{1}{2}\left(\mathcal{V}_{x_1}^{2}+\mathcal{V}_{x_2}^{2}+\mathcal{V}_y^{2}\right)=P(x_1,x_2,y,t)\,\,\rm{on}\,\,y=\eta(x_1,x_2,t),
\end{equation}
where $g$ is the gravitational constant, and $P(x_1,x_2,y,t)$
prescribed over the region of disturbance.
\\
at $t=0$.\\
{\bf Previous work}\\
For the 2 dimensional case when the motion of the free surface is a small perturbation of still water and without surface tension, we refer to Nalimov
\cite{Nal}.\\
In the case of the dock problem the upper surface of the water is constrained by the dock for all $x_2<0$ and is a free surface described by $y=\eta(x_1,x_2,t)$, subject to atmospheric presure for all $x_2>0$.
The standing solution of the homogeneous $(P=0)$ two-dimensional
dock problem has been given by Friedrichs and Lewy \cite{Frie2}, (see also \cite{Frie1}) as
a special case of periodic waves on sloping beaches which behave at
infinity like an arbitrary progressing wave. The general
three-dimensional case of periodic waves cresting on a beach sloping
at any angle $\alpha$ was first considered by Peters \cite{Pet} and
Roseau \cite{Ros}. The case of the three-dimensional dock problem in
water of uniform depth was first solved by Heins \cite{Hei} by means of the
Wiener-Hopf technique, see also Holford \cite{Hol}, Varley \cite{Var1}, \cite{Var2},  Rahimizadeh \cite{Rahi} and the paper \cite{Chak}.\\
{\bf Outline of the paper}\\
Unlike the situation described above, our paper will deal with a
problem in a two dimensional sector with a corner point. Instead of
the space variables $(x_1,x_2,y)$ we will have $(x_1,x_2)$. With
this notation the equation of the free surface
$\eqref{freesurfaceeq}$ now becomes
\begin{equation}
\label{twodimfreesurfaceeq}
x_2 =\eta (x_1,t).
\end{equation}
We shall denote by $G$ the corner domain in $\R^{2}$ formed by the
semi-axis $\Gamma_{1}=\{x_{1}>0, x_{2}=0\}$ and
$\Gamma_{2}=\{y_1=-x_{1}\cos\alpha-x_{2}\sin\alpha<0,
y_{2}=x_{1}\sin\alpha -x_{2}\cos\alpha=0\}$, where $\alpha$
represents the interior angle of $G$ and $0<\alpha\leq\pi$. The case $\alpha =\pi$ which corresponds to the dock problem was treated using different methods in \cite{Rahi}, see also \cite{Var1}.
\\
Let $v(x_1,x_2,t)$ denote the velocity potential function. We will
work under the assumption that the amplitude of the surface waves is
small with respect to the wave length. This will allow us to neglect
the nonlinear terms in $\eqref{dyncond}$. The assumption about small
amplitudes of the waves transforms the kinematic free surface
condition $\eqref{kincond}$ into $\eta_t-v_{x_2}=0$. These
considerations lead to the following linearized boundary value
problem
\begin{equation}
\label{linbvp}
\left\{\begin{array}{ccl} \Delta v(x_1,x_2,t) & = & 0\,\,\rm{for}\,\, (x_1,x_2)\in G\,\,\rm{and}\,t\geq 0\\
g\eta (x_1,t)+ v_t (x_1,0,t) & = & P(x_1,0,t)\\
\eta_t (x_1,t)-v_{x_2}(x_1,0,t) & = & 0,
\end{array}\right.
\end{equation}
subject to the initial conditions
\begin{equation}
\label{inc}
\left\{\begin{array}{ccl}\eta (x_1,0)  & = & \eta_0\\
v(x_1,x_2,0) & = & v_0\end{array}\right.
\end{equation}
with given $\eta_0$ and $v_{0}$ in appropriate Sobolev spaces.\\
In section $\ref{hypereveqconcl}$ we show existence and uniqueness for the problem $\eqref{linbvp}$ with initial conditions $\eqref{inc}$, namely we prove the following theorem
\begin{Thm}
\label{mainthm}
For any $T>0$ and for any $P(x_1,x_2,t)$ such that $P(x_1,0,t)\in C([0,T],L_{2}(\Gamma_{1}))$ and $P_{t}(x_1,0,t)\in L_{1}([0,T],L_{2}(\Gamma_{1}))$ there exist unique $v(x_1,x_2,t)\in C([0,T], \dot{H}_{1}(G))$ and $\eta(x_1,t)\in C([0,T],L_{2}(\Gamma_1))$ such that $v(x_1,0,t)\in C([0,T],H_{\frac{1}{2}}(\Gamma_1)), v_{t}(x_1,0,t)\in C([0,T],L_{2}(\Gamma_1)), \eta_t\in C([0,T],H_{-\frac{1}{2}}(\Gamma_1))$ which satisfy the boundary value problem

\begin{equation}
\label{finalbvp}
\left\{\begin{array}{ccr} \Delta v(x_1,x_2,t) & = & 0\,\,\,\, (x_1,x_2)\in G\,\,\rm{and}\,t\geq 0\\
g\eta (x_1,t)+ v_t (x_1,0,t) & = & P(x_1,0,t)\,\,\,\, (x_1,0)\in \Gamma_1\,\,\rm{and}\,t\geq 0\\
\eta_t (x_1,t)-v_{x_2}(x_1,0,t) & = & 0\,\,\,\, (x_1,0)\in \Gamma_1\,\,\rm{and}\,t\geq 0,
\end{array}\right.
\end{equation}
with the initial conditions
\begin{equation}
\left\{\begin{array}{ccl}\eta (x_1,0)  & = & \eta_0(x_1)\\
v(x_1,x_2,0) & = & v_0(x_1,x_2)\end{array}\right.
\end{equation}
where $\eta_{0}\in L_{2}(\Gamma_1)$ and $v_{0}\in \dot{H}_{1}(G)$.
\end{Thm}
In section $\ref{ellprob}$ we consider an auxiliary boundary value problem for which we show existence and uniqueness. This will ensure that the Dirichlet to Neumann operator is well defined. Another fact to be established is the selfadjointness and the positivity of the Dirichlet to Neumann operator which is done in section $\ref{dnselfadj}$.

\section{The elliptic problem in a corner}
\label{ellprob}
Consider the following auxiliary boundary value problem:
\begin{equation}
\label{auxbvp}
\left\{\begin{array}{ccl} \Delta v(x_1,x_2,t)& = &
\frac{\partial ^{2}v}{\partial x_1^{2}}(x_1,x_2,t)+ \frac{\partial
^{2}v}{\partial
x_2^{2}}(x_1,x_2,t) =  0\,\,\rm{for}\,\,(x_1,x_2)\in G, t\geq 0,\\
v|_{\Gamma_1}& = & f\\
\frac{\partial v}{\partial\nu}|_{\Gamma_2} & = & 0.
\end{array}\right.
\end{equation}
\\
\\
Theorem $\ref{eumixedbvp}$ will show that for $f$ in a suitable
Sobolev space we obtain a unique solution $v$ for the boundary value
problem $\eqref{auxbvp}$. This allows us to define by $$\Lambda
f=\frac{\partial v}{\partial x_2}\big|_{\Gamma_1}$$ the so called
Dirichlet to Neumann operator. Using $\Lambda$ and that
$f=v(x_1,0,t)$ we see that the boundary conditions in
$\eqref{linbvp}$ are equivalent to the single equation
$$v_{tt}+g\Lambda v=P,$$
for which existence and uniqueness will be proved.\\
The above boundary value problem $\eqref{auxbvp}$ will
be studied in classes of Sobolev spaces, cf. for example
\cite{Esk1}.We briefly recall those definitions. As usual
$H_s(\R^{2})$ denotes the Sobolev space with the norm
$$\vert\vert
u\vert\vert_{s}^{2}=\int_{\R^{2}}(1+\vert\xi\vert^{2})^{s}\vert\tilde{u}(\xi)\vert^{2}d\xi,$$
where $\tilde{u}$ is the Fourier transform of $u$. By
$\mathring{H}_s(G)$ we denote the subspace of $H_{s}(\R^{2})$
consisting of functions with support in $\overline{G}$. $H_{s}(G)$
is defined to be the space of all restrictions of functions in
$H_s(\R^{2})$ to the domain $G$ with the norm
\begin{equation}
\vert\vert f\vert\vert_{s}^{+ }=\inf _{l}\vert\vert
lf\vert\vert_{s},
\end{equation}
 where $f$ is a distribution in $G$, $lf$ is an
arbitrary extension of $f$ to $\R^{2}$ belonging to $H_{s}(\R^{2})$,
and the infimum is taken over all extensions of $f$.\\
On $\Gamma_k$, $k=1,2$, we define $H_s(\Gamma_k)$ to be the space of
all restrictions of distributions in $H_s(\R^1)$ to $\Gamma_k$ with
the norm
\begin{equation}
\label{gammanorm} [h]_s^{+}=\inf_{l}[lh]_s,
\end{equation}
where $lh$ is an arbitrary extension of $h$ to $\R^1$ and $[lh]_s$
is the norm in $H_{s}(\R^{1})$.\\
The Sobolev space $\mathring{H}_s(\Gamma_k)$ is defined as the
completion of $C_0^{\infty}(\Gamma_k)$ with respect to the norm
$\eqref{gammanorm}$.\\
We shall also need the following modifications of the Sobolev
spaces. Let us denote with $\dot{H}_s(\R^2)$ the closure of
$C_0^{\infty}(\R^{2})$ with respect to the norm
\begin{equation}
\label{dotsobspaces}
\int_{\R^2}\vert\xi\vert^{2s}\vert\tilde{u}(\xi)\vert^2
d\xi,\,\,\,\,\,\vert\xi\vert=\sqrt{\xi_1^2+\xi_2^2}
\end{equation}
Then, for the domain $G$ we define $\dot{H}_s(G)$ to be the space of
restrictions of distributions from $\dot{H}_s(\R^{2})$.
We will prove the following
\begin{Thm}
\label{eumixedbvp} For any
$f\in\mathring{H}_{\frac{1}{2}}(\Gamma_1)$  there exists a unique
$v(x_1,x_2)\in \dot{H}_1(G)$ such that
\begin{equation}
\label{firstpart} \left\{\begin{array}{ccl}
  \Delta v(x_1,x_2,t)& = & 0\,\,\rm{for}\,\,(x_1,x_2)\in G,t\geq 0,\\
v|_{\Gamma_1} & = & f,\\ \frac{\partial v}{\partial\nu}|_{\Gamma_2}
& = & 0.
\end{array}\right.
\end{equation}
\end{Thm}
This Theorem will be proved by showing that it is equivalent with
another boundary value problem whose existence and uniqueness are
proved in the paper \cite{Esk3}. We need first to fix some
notations.\\
{\bf Notation}
For $t\in\R$ we denote $$(t-i0)^{\frac{1}{2}-s}=\lim_{\varepsilon\to
0}e^{(\frac{1}{2}-s)\ln(t-i\varepsilon)},\,\,\varepsilon >0,$$ where
we take the branch of $\ln(t-i\varepsilon)$ that is real for $t>0$
and $\varepsilon =0$. Now, let
$$\Lambda_{-}^{\frac{1}{2}-s}=\left(i\frac{\partial}{\partial_{x_1}}\cos\frac{\alpha}{2}+i\frac{\partial}{\partial
x_2}\sin\frac{\alpha}{2}-i0\right)^{\frac{1}{2}-s}$$ be a
pseudodifferential operator in $\R^{2}$ with symbol
$$\Lambda_{-}^{\frac{1}{2}-s}(\xi_1,\xi_2)=\left(\xi_1\cos\frac{\alpha}{2}+\xi_2\sin\frac{\alpha}{2}-i0\right)^{\frac{1}{2}-s}$$

\begin{Rem}
The operator $\Lambda_{-}^{\frac{1}{2}-s}$ has the property that if
$u_{-}$ is a distribution with support in $CG$ then the support of
$\Lambda_{-}^{\frac{1}{2}-s}u_{-}$ is also in $CG$. An operator with
such a property is called a ``minus'' operator with respect to the
domain $G$. For proofs and details concerning ``minus'' operators
see Lemma $20.2$ in $\cite{Esk1}$ and Lemma 2.2 in \cite{Esk2}.
\end{Rem}
\begin{Rem}
\label{imppropmin} If $A_{-}$ is a ``minus'' operator and $u$ is a
distribution in $G$ we have that $p_{G}A_{-}lu$ is independent of
the choice of the extension $lu$ of $u$ to $\R^{2}$ where $p_{G}$ is
the restriction operator to $G$.
\end{Rem}
Remark $\ref{imppropmin}$ allows us to consider the following
boundary value problem:
\begin{equation}
\label{plusminbvp}
\left\{\begin{array}{lll} \Delta u & = & 0,\,\,(x_1,x_2)\in G\\
p_{1}^{+}\left(i\frac{\partial}{\partial
x_1}-i0\right)^{s-m_1-\frac{1}{2}}B_1\left(i\frac{\partial}{\partial_{x_1}},i\frac{\partial}{\partial_{x_2}}\right)\Lambda_{-}^{\frac{1}{2}-s}lu
& = &h_{1}(x_1),\,x_1>0\\

p_{2}^{-}\left(-i\frac{\partial}{\partial_{x_1}}\cos\alpha
-i\frac{\partial}{\partial_{x_2}}\sin\alpha+i0\right)^{s-m_2-\frac{1}{2}}B_2\left(i\frac{\partial}{\partial_{x_1}},i\frac{\partial}{\partial_{x_2}}\right)\Lambda_{-}^{\frac{1}{2}-s}lu
& = & h_2(y_1),\,y_1<0
\end{array}\right.
\end{equation}
where $p_{1}^{+}, p_{2}^{-}$ are restrictions operators to
$\Gamma_1,\Gamma_2$, respectively, $B_1(\xi_1,\xi_2),
B_2(\xi_1,\xi_2)$ are homogeneous polynomials of degrees $m_1, m_2$
respectively, and the coordinates $(y_1,y_2)$ are related to
$(x_1,x_2)$  through the equations

\begin{equation}
\label{coordch}
y_1=-x_1\cos\alpha-x_2\sin\alpha,\,\,\,y_2=x_1\sin\alpha-x_2\cos\alpha.
\end{equation}
\begin{Rem}
Theorem $2.1$ from $\cite{Esk3}$ asserts that for any $(h_1,h_2)\in
L^{2}(\Gamma_1)\times L^{2}(\Gamma_2)$ the boundary value problem
$\eqref{plusminbvp}$ has a unique solution
$u\in\dot{H}_{\frac{1}{2}}(G)$ provided $s$ satisfies the so-called
``corner condition'' $(2.73)$ from $\cite{Esk3}$. If $B_1$ is now
the identity operator and
$B_2=\frac{\partial}{\partial\nu}=\nu_1\frac{\partial}{\partial
x_1}+\nu_2\frac{\partial}{\partial x_2}$, a calculation shows that
$s=1$ verifies the ``corner condition'' $(2.73)$ from $\cite{Esk3}$.
We only state what this conditions means in our case and show that
$s=1$ verifies it. For details we ask the
reader to consult the proof in $\cite{Esk3}$.\\
\\

Let us first establish some notations. We set first
$\lambda_1=-i,\lambda_2=i$. Denote by
$$\mu_j=\frac{\sin\alpha-\lambda_j\cos\alpha}{-\cos\alpha-\lambda_j\sin\alpha},\,\,j=1,2,$$
from which follows that $\mu_1=-i$ and $\mu_2=i$. We will also need
the numbers $\beta_1,\beta_2$ given by
\begin{equation}
\label{betas} i\beta_k=\ln
(\cos\alpha+\lambda_k\sin\alpha)=\ln\vert\cos\alpha+\lambda_k\sin\alpha\vert+i\rm{arg}(\cos\alpha+\lambda_k\sin\alpha),\,k=1,2.
\end{equation}
from which we deduce that $\beta_1=2\pi-\alpha$ and
$\beta_2=\alpha$.\\
Denote also by $B_{2}^{(1)}(\eta_1,\eta_2)$ the symbol of $B_2$ in
the $(y_1,y_2)$ coordinates. It follows from $\eqref{coordch}$ that
it has the form
$$B_2^{(1)}(\eta_1,\eta_2)=B_2(-\eta_1\cos\alpha+\eta_2\sin\alpha,-\eta_1\sin\alpha-\eta_2\cos\alpha)=
i\nu_1(-\eta_1\cos\alpha+\eta_2\sin\alpha)+i\nu_2(-\eta_1\sin\alpha-\eta_2\cos\alpha).$$
We can now formulate the corner condition, which is the following:

\begin{equation}
\label{cornercondition} M_{0}\left(z-s+\frac{1}{2}\right)\neq
0,\,\,\,\,\textup{for any}\,\,\,z=\frac{1}{2}+i\tau,\,\,\tau\in\R
\end{equation}
where $$M_{0}(z)=-b_{2}^{(0)}+e^{2\pi iz}e^{-i\beta_1 z}e^{i\beta_2
z},$$  $$b_2^{(0)}=(B_2^{+})^{-1}e^{i\pi}B_2^{-},$$
$$B_2^{+}=B_2^{(1)}(1,\mu_1)=B_2^{(1)}(1,-i),$$
$$B_2^{-}=B_2^{(1)}(-1,-\mu_2)=B_2^{(1)}(-1,-i).$$
We will show that $s=1$ verifies the condition
$\eqref{cornercondition}$, i.e., we will show that $M_{0}(i\tau)\neq
0$ for all $\tau\in\R$. We have the following
$$B_2^{+}=\nu_1\sin\alpha-\nu_2\cos\alpha-i(\nu_2\sin\alpha+\nu_1\cos\alpha),$$
$$B_2^{-}=\nu_1\sin\alpha-\nu_2\cos\alpha+i(\nu_2\sin\alpha+\nu_1\cos\alpha),$$
equalities which show that $\vert b_2^{(0)}\vert =1$. Since
$\beta_1=2\pi -\alpha$ and $\beta_2=\alpha$ we obtain that
\begin{equation}
\label{simpcornercond}
M_{0}(i\tau)=-b_2^{(0)}+e^{-2\alpha\tau}.
\end{equation}
Since $\alpha >0$, we see from equation $\eqref{simpcornercond}$
that $M_{0}(i\tau)\neq 0$ for all $\tau\neq 0$. Therefore it only
remains to show that $b_{2}^{(0)}\neq 1$.\\
Now
\begin{equation}
\begin{split}
b_{2}^{(0)}&=-\frac{B_2{-}}{B_{2}^{+}}=-\frac{\nu_1\sin\alpha-\nu_2\cos\alpha+i(\nu_2\sin\alpha+\nu_1\cos\alpha)}
{\nu_1\sin\alpha-\nu_2\cos\alpha-i(\nu_2\sin\alpha+\nu_1\cos\alpha)}\\
&
=\frac{(\nu_1^{2}-\nu_2^{2})\cos(2\alpha)+2\nu_1\nu_2\sin(2\alpha)+\left(2\nu_1\nu_2\cos(2\alpha)-(\nu_1^{2}-\nu_2^{2})\sin(2\alpha)\right)i}{\nu_1^{2}+\nu_2^{2}}\\
& =-\cos^{2}(2\alpha)-\sin^{2}(2\alpha)=-1,
\end{split}
\end{equation}
since $(\nu_1,\nu_2)=(-\sin\alpha,\cos\alpha)$. Thus $s=1$ verifies
the corner condition $\eqref{cornercondition}$.
\end{Rem}
Therefore the following Theorem is true and allows us to prove Theorem $\ref{eumixedbvp}$.
\begin{Thm}
\label{thmpartplusminbvp} The boundary value problem

\begin{equation}
\label{partplusminbvp}
\left\{\begin{array}{lll} \Delta u & = & 0,\,\,(x_1,x_2)\in G\\
p_{1}^{+}\left(i\frac{\partial}{\partial
x_1}-i0\right)^{\frac{1}{2}}\Lambda_{-}^{-\frac{1}{2}}lu
& = &h(x_1),\,x_1>0\\

p_{2}^{-}\left(-i\frac{\partial}{\partial_{x_1}}\cos\alpha
-i\frac{\partial}{\partial_{x_2}}\sin\alpha+i0\right)^{-\frac{1}{2}}\frac{\partial}{\partial
\nu}\Lambda_{-}^{-\frac{1}{2}}lu & = & 0,\,y_1<0
\end{array}\right.
\end{equation}
has a unique solution $u\in\dot{H}_{\frac{1}{2}}(G)$ for any $h\in
L_2(\Gamma_1)$.
\end{Thm}
{\bf Proof of Theorem $\ref{eumixedbvp}$:}\\
We start with $f\in\mathring{H}_{\frac{1}{2}}(\Gamma_1)$. Put now
$$h=p_1^{+}\left(i\frac{\partial}{\partial x_1}-i0\right)^{\frac{1}{2}}f.$$ This implies that
$h\in L_2(\Gamma_1)$. Now we can apply Theorem
$\ref{thmpartplusminbvp}$ Let $u$ be the unique solution of boundary
value problem $\eqref{partplusminbvp}$.
We then set
\begin{equation}
\label{vasfcnofu}
v=p_{G}\Lambda_{-}^{-\frac{1}{2}}lu,
\end{equation}
where $p_G$ and $l$ are as before. Using Lemma $2.2$ from
\cite{Esk2} we obtain that $v\in\dot{H}_{1}(G)$. Due to the fact
that ``minus operators'' commute with the differential operators we
have that
\begin{equation}
\label{one} \Delta v=0,\,\, \rm{for}\,\,(x_1,x_2)\in G
\end{equation}
From $\eqref{vasfcnofu}$ and the second equation of
$\eqref{partplusminbvp}$ we see that
$$p_1^{+}\left(i\frac{\partial}{\partial
x_1}-i0\right)^{\frac{1}{2}}v=h=
p_1^{+}\left(i\frac{\partial}{\partial
x_1}-i0\right)^{\frac{1}{2}}f.$$ It then follows that
$$p_1^{+}\left(i\frac{\partial}{\partial
x_1}-i0\right)^{\frac{1}{2}}(v-f)=0$$ which means that
$$\left(i\frac{\partial}{\partial
x_1}-i0\right)^{\frac{1}{2}}(v-f)=v_{-},$$ where $v_{-}$ has its
support in $\R^{1}\setminus\Gamma_1$. But then
$$v-f=\left(i\frac{\partial}{\partial
x_1}-i0\right)^{-\frac{1}{2}}(v_{-})$$ and since
$\left(i\frac{\partial}{\partial x_1}-i0\right)^{-\frac{1}{2}}$ is a
``minus operator'' we obtain that the support of $v-f$ is contained
in $\R^{1}\setminus\Gamma_1$. This just means that
\begin{equation}
\label{two} v|_{\Gamma_1}=f.
\end{equation}
Using the second boundary condition in $\eqref{partplusminbvp}$ and
the fact that
$$p_{2}^{-}\left(-i\frac{\partial}{\partial_{x_1}}\cos\alpha
-i\frac{\partial}{\partial_{x_2}}\sin\alpha+i0\right)^{\frac{1}{2}}$$
is a ``minus operator'', we obtain that
\begin{equation}
\label{three}
 \frac{\partial v}{\partial\nu}|_{\Gamma_2}=0.
\end{equation}
The relations $\eqref{one},\eqref{two}$, and $\eqref{three}$ prove
the existence of a solution to the boundary value problem
$\eqref{firstpart}$. In order to prove the uniqueness we will show
that the boundary value problem
\begin{equation}
\label{homogbvp} \left\{\begin{array}{ccl}
  \Delta v(x_1,x_2,t)& = & 0\,\,\rm{for}\,\,(x_1,x_2)\in G,t\geq 0,\\
v|_{\Gamma_1} & = & 0,\\ \frac{\partial v}{\partial\nu}|_{\Gamma_2}
& = & 0.
\end{array}\right.
\end{equation}
has only the trivial solution in $\dot{H}_{1}(G)$. Let $v \in \dot{H}_{1}(G)$ be a
solution of the boundary value problem $\eqref{homogbvp}$. Denote by
$lv$ the extension by zero of $v$ to $\R^{2}$ and put
$u=p_{G}\Lambda_{-}^{\frac{1}{2}}lv$. Then
$u\in\mathring{H}_{\frac{1}{2}}(G)$ and $u$ satisfies the following
boundary value problem
\begin{equation}
\left\{\begin{array}{lll} \Delta u & = & 0,\,\,(x_1,x_2)\in G\\
p_{1}^{+}\left(i\frac{\partial}{\partial
x_1}-i0\right)^{\frac{1}{2}}\Lambda_{-}^{-\frac{1}{2}}lu
& = &0,\,x_1>0\\
p_{2}^{-}\left(-i\frac{\partial}{\partial_{x_1}}\cos\alpha
-i\frac{\partial}{\partial_{x_2}}\sin\alpha+i0\right)^{-\frac{1}{2}}\frac{\partial}{\partial
\nu}\Lambda_{-}^{-\frac{1}{2}}lu & = & 0,\,y_1<0
\end{array}\right.
\end{equation}
Since by Theorem $\ref{thmpartplusminbvp}$ the solution to the above
boundary value problem is unique, it follows that $u=0$. Using that
$u=p_{G}\Lambda_{-}^{\frac{1}{2}}lv$ and that
$\Lambda_{-}^{\frac{1}{2}}$ is a ``minus operator'' it follows that
$v=0$.

\begin{Rem}
\label{defDN} Theorem $\ref{eumixedbvp}$ gives rise to an operator
$$\Lambda :\mathring{H}_{\frac{1}{2}}(\Gamma_1)\rightarrow
H_{-\frac{1}{2}}(\Gamma_1)$$ defined by
\begin{equation}
\label{formulaDN} \Lambda f:=\frac{\partial
v}{\partial\nu}\big|_{\Gamma_1},
\end{equation}
where $v\in\dot{H}_1(G)$ is the
unique solution to the boundary value problem $\eqref{firstpart}$.
$\Lambda$ is called the Dirichlet to Neumann operator.
\end{Rem}

\section{Selfadjointness of the Dirichlet to Neumann operator}
\label{dnselfadj}
We now return to the system
\begin{equation}
\label{etasys}
\left\{\begin{array}{ccl}
g\eta (x_1,t)+ v_t (x_1,0,t) & = & P(x_1,0,t)\\
\eta_t (x_1,t)-v_{x_2}(x_1,0,t) & = & 0
\end{array}\right.
\end{equation} at $x_2=0$, with initial conditions
\begin{equation}
\left\{\begin{array}{ccl}\eta(x_1,0)& = &\eta_0(x_1)\\
 v(x_1,0,0) &  = & v_0(x_1,0)
\end{array}\right.
\end{equation}
By elimination of $\eta$ between the above relations and using
$\eqref{formulaDN}$ the single equation in $v$ is obtained:
\begin{equation}
\label{etaelim}
 v_{tt}(x_1,0,t)+g\Lambda v(x_1,0,t)=P_{t}(x_1,0,t),
\end{equation}
with initial conditions
\begin{equation}v(x_1,0,0)=v_{0}(x_1,0),\,\,\,\,v_{t}(x_1,0,0)=P(x_1,0,0)-g\eta_{0}(x_1)\end{equation}
By denoting $\eta =\frac{P}{g}-\frac{v_t}{g}$ we see that
$\eta_{t}=\frac{P_t}{g}-\frac{v_{tt}}{g}$ and from $\eqref{etaelim}$
we have that $\eta_t-v_{x_2}=0$, and this shows that
$\eqref{etasys}$ and $\eqref{etaelim}$ are equivalent. Therefore we
will show existence and uniqueness for $\eqref{etaelim}$. In order
to do this we will show that the operator $\Lambda$ is a positive
and self-adjoint operator. We consider first $\tilde{\Lambda}$ to be
the unbounded operator with domain
$$\rm{dom}(\tilde{\Lambda})=\{f\in C_{0}^{\infty}(\Gamma_1)\,\rm{such\,
that}\,\tilde{\Lambda} f\in L_{2}(\Gamma_1)\},$$ and defined also by
$\eqref{formulaDN}$.

\begin{Thm}
We have that $(\tilde{\Lambda} f,g)=(f,\tilde{\Lambda} g)$ for every
$f,g\in \rm{dom}(\tilde{\Lambda})$, i.e., $\tilde{\Lambda}$ is a
symmetric operator.
\end{Thm}

\begin{proof}
Let $f, g\in \rm{dom}(\tilde{\Lambda})$. Let $v\in\dot{H}_1(G)$ be
the unique solution of the boundary value problem
\begin{equation}
 \left\{\begin{array}{ccl}
  \Delta v& = & 0\,\,\rm{in}\,\, G,\\
v|_{\Gamma_1} & = & f,\\ \frac{\partial v}{\partial\nu}|_{\Gamma_2}
& = & 0
\end{array}\right.
\end{equation}
cf. Theorem $\ref{eumixedbvp}$.\\
Let also $u\in\dot{H}_1(G)$ be
the unique solution of the boundary value problem
\begin{equation}
 \left\{\begin{array}{ccl}
  \Delta u& = & 0\,\,\rm{in}\,\, G,\\
u|_{\Gamma_1} & = & g,\\ \frac{\partial u}{\partial\nu}|_{\Gamma_2}
& = & 0
\end{array}\right.
\end{equation}
cf. Theorem $\ref{eumixedbvp}$.\\
Let $\varepsilon >0,\,N>0$ be arbitrary positive numbers. We will
now apply the first Green formula for the domain
$$G_{\varepsilon N}:=\{(r,\theta):\varepsilon\leq r\leq
N,\,0\leq\theta\leq\alpha\},$$ and for the functions $u$ and $v$.
Let us also denote
$C_{\varepsilon}:=\{(\varepsilon,\theta):0\leq\theta\leq\alpha\}$
and $C_{N}:=\{(N,\theta):0\leq\theta\leq\alpha\}$. We then have

\begin{equation}
\int\!\int_{G_{\varepsilon N}}v\Delta u\, dx_1 dx_2
=-\int\!\int_{G_{\varepsilon N}}\nabla u\nabla v\, dx_1 dx_2
+\int_{C_{\varepsilon}}v\frac{\partial u}{\partial\nu}d\sigma
+\int_{C_{N}}v\frac{\partial u}{\partial\nu}d\sigma
+\int_{\varepsilon}^{N}v\frac{\partial u}{\partial\nu}dx_1
\end{equation}
Since $\Delta u=0$ the last equation becomes
\begin{equation}
\label{greenformula}
 \int\!\int_{G_{\varepsilon
N}}\nabla u\nabla v\, dx_1
dx_2=\int_{C_{\varepsilon}}v\frac{\partial u}{\partial\nu}d\sigma
+\int_{C_{N}}v\frac{\partial u}{\partial\nu}d\sigma
+\int_{\varepsilon}^{N}v\frac{\partial u}{\partial\nu}dx_1
\end{equation}
for every $\varepsilon >0$ and every $N>0$. We are going to prove
that
\begin{equation}
\label{Ntoinfty} \lim_{N\to\infty}\int_{C_{N}}v\frac{\partial
u}{\partial\nu}d\sigma =0.
\end{equation}
and that
\begin{equation}
\label{epstozero}
 \lim_{\varepsilon\to 0}\int_{C_{\varepsilon}}v\frac{\partial u}{\partial\nu}d\sigma =0.
\end{equation}
We now pass to polar coordinates $(r,\theta)$ and perform the standard procedure of separation of variables. We obtain that the general
solutions to the equation $\Delta v=0$ are of the form
$$v(r,\theta)=(A\cos(\sqrt{\lambda}\theta)+ B\sin(\sqrt{\lambda}\theta))(Cr^{\sqrt{\lambda}}+Dr^{-\sqrt{\lambda}})$$

We now exploit the boundary condition on $\Gamma_1$ and $\Gamma_2$.\\
From
\begin{equation}
\label{chtotrig1} \frac{\partial v}{\partial
x_1}=\cos\theta\frac{\partial v}{\partial
r}-\frac{\sin\theta}{r}\frac{\partial
v}{\partial\theta}\end{equation} and
\begin{equation}
\label{chtotrig2} \frac{\partial v}{\partial
x_2}=\sin\theta\frac{\partial v}{\partial
r}+\frac{\cos\theta}{r}\frac{\partial v}{\partial\theta}
\end{equation}
we obtain that
\begin{equation}
\label{normalcondongamma2}
\frac{\partial
v}{\partial\nu}\big|_{\Gamma_2}=-\sin\alpha\frac{\partial
v}{\partial x_1}\big|_{\theta=\alpha}+\cos\alpha\frac{\partial
v}{\partial
x_2}\big|_{\theta=\alpha}=\frac{\sin^{2}\alpha}{r}\frac{\partial
v}{\partial\theta}\big|_{\theta=\alpha}+\frac{\cos^{2}\alpha}{r}\frac{\partial
v}{\partial\theta}\big|_{\theta=\alpha}=\frac{1}{r}\frac{\partial
v}{\partial\theta}\big|_{\theta=\alpha}.
\end{equation}
Since $f\in C_{0}^{\infty}(\Gamma_1)$ there exist $\varepsilon _{0} >0$ and $N_{0}>0$ such that $v\big|_{\Gamma_1}=0$ for $x_1<\varepsilon_{0}$ and $v\big|_{\Gamma_1}=0$ for $x_1>N_{0}$.
From the condition $\frac{\partial v}{\partial\nu}\big|_{\Gamma_2}=0$ we obtain utilizing $\eqref{normalcondongamma2}$ the separated solutions
\begin{equation}
v_n(r,\theta)=B_n\sin
\left(n+\frac{1}{2}\right)\frac{\pi}{\alpha}\theta \left(C_{n}
r^{(n+\frac{1}{2})\frac{\pi}{\alpha}}+D_{n}r^{-(n+\frac{1}{2})\frac{\pi}{\alpha}}\right).
\end{equation}
Since $v\in\dot{H}_1(G)$ it follows that there exists an integer $n_{0}>0$
such that $v$ has the representation
\begin{equation}
\label{vinseppol} v(r,\theta)=\sum_{n=-\infty}^{n_0}A_n
r^{(n+\frac{1}{2})\frac{\pi}{\alpha}}   \sin
\left(n+\frac{1}{2}\right)\frac{\pi}{\alpha}\theta
,\,\,\rm{for}\,\,\,\,r\geq N_{0}
\end{equation}
and $n_0$ is to be determined from the
condition $$\int\int_{G} \vert\nabla u(x_1,x_2)\vert^{2}\,dx_1 dx_2
<\infty.$$
Using relations $\eqref{chtotrig1}$ and
$\eqref{chtotrig2}$ the last condition is written in polar
coordinates as
\begin{equation}
\label{condinrtheta}
\int_{0}^{\alpha}\int_{N_{0}}^{\infty}\left(\left(\frac{\partial
v}{\partial  r}\right)^{2}+\frac{1}{r^{2}}\left(\frac{\partial
v}{\partial \theta}\right)^{2}\right)r\,dr\,d\theta <\infty.
\end{equation}
From the representation $\eqref{vinseppol}$ of $v$ we obtain, using
$\eqref{condinrtheta}$, the condition
$$\int_{N_{0}}^{\infty}r^{2\left[\left(n_{0}+\frac{1}{2}\right)\frac{\pi}{\alpha}-1\right]}\left(1+O\left(\frac{1}{r}\right)\right)r\,dr<\infty,$$
which is satisfied if and only if
$2\left(n_{0}+\frac{1}{2}\right)\frac{\pi}{\alpha}-2+1<-1$, which is
equivalent to $n_{0}<-\frac{1}{2}$. Since $n_{0}$ is an integer we
have that $\eqref{condinrtheta}$ is satisfied if and only if
$n_{0}\leq -1$.\\
In order to prove that
\begin{equation}
\label{Ntoinfty} \lim_{N\to\infty}\int_{C_{N}}v\frac{\partial
u}{\partial\nu}d\sigma =0.
\end{equation}
we note first that  $$\int_{C_{N}}v\frac{\partial
u}{\partial\nu}d\sigma=\int_{0}^{\alpha}v(N,\theta)\frac{\partial
u}{\partial r}(N, \theta)N\,d\theta.$$
Since
\begin{equation}
\label{oform}
\begin{array}{ccl}
v(r,\theta) & = & \sum_{n=-\infty}^{-1}A_n
r^{(n+\frac{1}{2})\frac{\pi}{\alpha}}   \sin
\left(n+\frac{1}{2}\right)\frac{\pi}{\alpha}\theta\\
 & = & r^{-\frac{\pi}{2\alpha}}\left(A_{-1}\sin\left(\frac{-\pi}{2\alpha}\theta\right)+O\left(r^{\frac{-\pi}{\alpha}}
\right)\right)
\end{array}
\end{equation}
and since a formula like $\eqref{oform}$ is true for $u$ it suffices
to show that
$$\lim_{N\to\infty}\int_{0}^{\alpha}N^{-\frac{\pi}{2\alpha}}N^{-\frac{\pi}{2\alpha}-1}N\,d\theta
=0.$$ The last equality is obviously true and therefore the equality
$\eqref{Ntoinfty}$ is proved.\\
Our next task is to prove that
\begin{equation}
\label{epstozero}
 \lim_{\varepsilon\to 0}\int_{C_{\varepsilon}}v\frac{\partial u}{\partial\nu}d\sigma =0.
\end{equation}
First of all note that
\begin{equation}
\label{lineintzero} \int_{C_{\varepsilon}}v\frac{\partial
u}{\partial\nu}d\sigma=\int_{0}^{\alpha}v(\varepsilon,\theta)\frac{\partial
u}{\partial r}(\varepsilon, \theta)\varepsilon\,d\theta
\end{equation}
Using again that $v\in \dot{H}_{1}(G)$ it follows that there exists
an $\varepsilon _{0} >0$ such that $v$ has the representation
\begin{equation}
\label{reprzero} v(r,\theta)=\sum_{n=n^{0}}^{\infty}A_n
r^{(n+\frac{1}{2})\frac{\pi}{\alpha}} \sin
\left(n+\frac{1}{2}\right)\frac{\pi}{\alpha}\theta,\,\,\rm{for}\,\,r\leq\varepsilon
_{0}
\end{equation}
where $n^0$ is some fixed integer which is to be determined from the
condition $$\int\int \vert\nabla u(x_1,x_2)\vert^{2}\,dx_1 dx_2
<\infty,$$ which in polar coordinates is written as
\begin{equation}
\label{condinpolzero}
\int_{0}^{\alpha}\int_{0}^{\varepsilon_{0}}\left(\left(\frac{\partial
v}{\partial r}\right)^{2}+\frac{1}{r^{2}}\left(\frac{\partial
v}{\partial \theta}\right)^{2}\right)r\,dr\,d\theta <\infty.
\end{equation}
From the representation $\eqref{reprzero}$ of $v$ and using
$\eqref{condinpolzero}$ we obtain the condition

$$\int_{0}^{\varepsilon_{0}}r^{2\left[\left(n^{0}+\frac{1}{2}\right)\frac{\pi}{\alpha}-1\right]}\left(1+O\left(r\right)\right)r\,dr<\infty,$$
which is satisfied if and only if
$2\left(n^{0}+\frac{1}{2}\right)\frac{\pi}{\alpha}-2+1>-1$, which is
equivalent to $n^{0}>-\frac{1}{2}$. Since $n^{0}$ is an integer we
have that $\eqref{condinpolzero}$ is satisfied if and only if
$n^{0}\geq 0$. Therefore,
\begin{equation}
\label{oformzero} \begin{array}{ccl} v(r,\theta)& =
&\sum_{n=0}^{\infty}A_n r^{(n+\frac{1}{2})\frac{\pi}{\alpha}} \sin
\left(n+\frac{1}{2}\right)\frac{\pi}{\alpha}\theta\\
 & = & r^{\frac{\pi}{2\alpha}}\left(
 A_{0}\sin\left( \frac{\pi}{2\alpha}\theta\right)+O\left( r^{\frac{\pi}{\alpha}}\right)
 \right),
\end{array}
\end{equation}
and a formula like $\eqref{oformzero}$ is also valid for $u$. In
order to prove $\eqref{epstozero}$ we use $\eqref{lineintzero}$,
$\eqref{oformzero}$ and therefore it suffices to show that
$$\lim_{\varepsilon\to
0}\int_{0}^{\alpha}\varepsilon^{\frac{\pi}{2\alpha}}\varepsilon^{\frac{\pi}{2\alpha}-1}\varepsilon\,d\theta
=0,$$ which is true.\\
Passing to the limit with $\varepsilon\to 0$ and $N\to \infty$ in
the formula $\eqref{greenformula}$ and using $\eqref{epstozero}$ and
$\eqref{Ntoinfty}$ we obtain that
\begin{equation}
\label{lambdag} \int\!\int_{G}\nabla u\nabla
v\,dx_{1}dx_{2}=\int_{\Gamma_1} v\frac{\partial
u}{\partial\nu}dx_{1}.
\end{equation}
Analogously we obtain that
\begin{equation}
\label{lambdaf} \int\!\int_{G}\nabla v\nabla
u\,dx_{1}dx_{2}=\int_{\Gamma_1} u\frac{\partial
v}{\partial\nu}dx_{1}.
\end{equation}
From $\eqref{lambdag}$ and $\eqref{lambdaf}$ we then see that
$(f,\tilde{\Lambda} g)=(g,\tilde{\Lambda} f)$ for every $f,g\in
\rm{dom}(\tilde{\Lambda})$.
\end{proof}
\begin{Cor}
\label{} For every non-zero $f\in\rm{dom}(\tilde{\Lambda})$ we have
that $(\tilde{\Lambda} f,f)> 0$.
\end{Cor}
\begin{proof}
Taking $v=u$ and using $u\big|_{\Gamma_1}=f$, $\frac{\partial
u}{\partial\nu}\big|_{\Gamma_1}=\tilde{\Lambda} f$ and
$\eqref{lambdaf}$ we obtain that $$(\tilde{\Lambda}
f,f)=\int\!\int_{G}\vert\nabla u\vert ^{2}\,dx_{1}dx_{2},$$ which
proves the claim.
\end{proof}
\section{The hyperbolic evolution equation on $\Gamma_{1}$ and the conclusion of the proof of the main theorem}
\label{hypereveqconcl}
\begin{Rem}
\label{Friedextuse}
Since
$\left(H_{-\frac{1}{2}}(\Gamma_1)\right)^{*}=\mathring{H}_{\frac{1}{2}}(\Gamma_1)$
we have that the Friedrichs extension of $\tilde{\Lambda}$ is
exactly the Dirichlet to Neumann operator $\Lambda$ defined by
$\eqref{formulaDN}$. Therefore $\Lambda$ is a positive self-adjoint
operator. The latter fact allows us to show existence and uniqueness
for our initial problem:
\begin{equation}
v_{tt}(x_1,0,t)+g\Lambda v(x_1,0,t)=P_{t}(x_1,0,t),
\end{equation}
with initial conditions
\begin{equation}v(x_1,0,0)=v_{0}(x_1,0),\,\,\,\,v_{t}(x_1,0,0)=P(x_1,0,0)-g\eta_{0}(x_1).\end{equation}
For simplicity, we denote $v_1(x_1):=P(x_1,0,0)-g\eta_{0}(x_1)$ and obtain the following initial problem:

\begin{equation}
\label{inhomogfin} v_{tt}(x_1,t)+g\Lambda
v(x_1,t)=P_{t}(x_1,t),\,\,\,\,\,v(x_1,0)=v_{0},\,\,\,\,\,v_{t}(x_1,0)=v_1.
\end{equation}
\end{Rem}
\begin{Lem}
The solution of the homogeneous problem
\begin{equation}
\label{homogfin} v_{tt}(x_1,t)+g\Lambda
v(x_1,t)=0,\,\,\,\,\,v(x_1,0)=v_{0},\,\,\,\,\,v_{t}(x_1,0)=v_1,
\end{equation}
is given by the formula
\begin{equation}
\label{formzeid}
v(x_1,t)=\cos(t\Lambda^{\frac{1}{2}})v_{0}+\underline{\Lambda}
^{-\frac{1}{2}}\sin(t\underline{\Lambda}^{\frac{1}{2}})v_{1},
\end{equation}
where $\underline{\Lambda}=g\Lambda$.
\end{Lem}
\begin{proof}
See for instance \cite{Zei}, pp. 309.
\end{proof}
For each $s\in\R$ let now $u(x_1,t;s)$ be the
solution of
$$u_{tt}+\underline{\Lambda}
u=0,\,\,\,u(x_1,0;s)=0,\,\,\,u_t(x_1,0;s)=P_{t}(x_1,0;s).$$
From $\eqref{formzeid}$ it follows that
\begin{equation}
\label{uoftands}
u(x_1,t,s)=\underline{\Lambda}^{-\frac{1}{2}}\sin(t\underline{\Lambda}^{\frac{1}{2}})P_{t}(x_1,0;s).
\end{equation}
We then have the following
\begin{Lem}
\label{Duhamel}
The function defined by $v(x_1,t)=\int_{0}^{t}u(x_1,t-s;s)\,ds$
satisfies the boundary value problem
\begin{equation}
\label{intinhomog} v_{tt}+g\underline{\Lambda}
v=P_{t}(x_1,t)\,\,\,\,v(x_1,0)=0,\,\,\,\,v_{t}(x_1,0)=0.
\end{equation}
\end{Lem}
\begin{proof}
Clearly $v(x_1,0)=0$. We also have
$$v_{t}(x_1,t)=u(x_1,0;t)+\int_{0}^{t}u_{t}(x_1,t-s;s)\,ds=\int_{0}^{t}u_{t}(x_1,t-s;s)\,ds,$$
which implies that $v_{t}(x_1,0)=0$. Finally, differentiating once
more in $t$ we obtain
\begin{equation}
\begin{split}
v_{tt}(x_1,t)& =u_{t}(x_1,0;t)+\int_{0}^{t}u_{tt}(x_1,t-s;s)\,ds\\
&=P_{t}(x_1,t)+\int_{0}^{t}-g\underline{\Lambda}
u(x_1,t-s;s)\,ds=P_{t}(x_1,t)-g\underline{\Lambda} v(x_1,t),
\end{split}
\end{equation}
which proves $\eqref{intinhomog}$.
\end{proof}
\begin{Cor}
\label{represformforv}
The solution of the problem $\eqref{inhomogfin}$ is given by the following formula
$$v(x_1,t)=\int_{0}^{t}\underline{\Lambda}^{-\frac{1}{2}}\sin((t-s)\underline{\Lambda}^{\frac{1}{2}})P_{t}(x_1,0;s)\,ds +\cos(t\underline{\Lambda}^{\frac{1}{2}})v_{0}+\underline{\Lambda}
^{-\frac{1}{2}}\sin(t\underline{\Lambda}^{\frac{1}{2}})v_{1}.$$
\end{Cor}
\begin{proof}
Adding up the solutions to the problems $\eqref{homogfin}$ and
$\eqref{intinhomog}$  and taking into account formula $\eqref{uoftands}$ we obtain the assertion.
\end{proof}
Remark $\ref{Friedextuse}$ and Lemma $\ref{Duhamel}$ allow us to conclude the proof of the main Theorem $\ref{mainthm}$. We restate it here for convenience.
\begin{Thm}
For any $T>0$ and for any $P(x_1,x_2,t)$ such that $P(x_1,0,t)\in C([0,T],L_{2}(\Gamma_{1}))$ and $P_{t}(x_1,0,t)\in L_{1}([0,T],L_{2}(\Gamma_{1}))$ there exist unique $v(x_1,x_2,t)\in C([0,T], \dot{H}_{1}(G))$ and $\eta(x_1,t)\in C([0,T],L_{2}(\Gamma_1))$ such that $v(x_1,0,t)\in C([0,T],H_{\frac{1}{2}}(\Gamma_1)), v_{t}(x_1,0,t)\in C([0,T],L_{2}(\Gamma_1)), \eta_t\in C([0,T],H_{-\frac{1}{2}}(\Gamma_1))$ which satisfy the boundary value problem
\begin{equation}
\label{finalbvp}
\left\{\begin{array}{ccr} \Delta v(x_1,x_2,t) & = & 0,\,\,\,\, (x_1,x_2)\in G\,\,\rm{and}\,t\geq 0\\
g\eta (x_1,t)+ v_t (x_1,0,t) & = & P(x_1,0,t),\,\,\,\, (x_1,0)\in \Gamma_1\,\,\rm{and}\,t\geq 0\\
\eta_t (x_1,t)-v_{x_2}(x_1,0,t) & = & 0,\,\,\,\, (x_1,0)\in \Gamma_1\,\,\rm{and}\,t\geq 0,
\end{array}\right.
\end{equation}
with the initial conditions
\begin{equation}
\left\{\begin{array}{ccl}\eta (x_1,0)  & = & \eta_0(x_1)\\
v(x_1,x_2,0) & = & v_0(x_1,x_2)\end{array}\right.
\end{equation}
where $\eta_{0}\in L_{2}(\Gamma_1)$ and $v_{0}\in \dot{H}_{1}(G)$.
\end{Thm}
\begin{proof}
We first prove the assertion about $v$. From Corollary $\ref{represformforv}$ we have that
$$\underline{\Lambda} ^{\frac{1}{2}}v(x_1,0,t)=\int_{0}^{t}\sin ((t-s)\underline{\Lambda}^{\frac{1}{2}})P_{t}(x_1,0;s)\,ds+\cos(t\underline{\Lambda}^{\frac{1}{2}})\underline{\Lambda}^{\frac{1}{2}}v_{0}+
\sin(t\underline{\Lambda}^{\frac{1}{2}})v_{1}.$$
Therefore we have that $$\max_{0\leq t\leq T}\vert\vert\underline{\Lambda}^{\frac{1}{2}}v\vert\vert_{0}\leq C\left(\int_{0}^{T}\vert\vert P_{t}\vert\vert_{0}\,dt +\vert\vert v_{0}\vert\vert_{\frac{1}{2}}+\vert\vert v_{1}\vert\vert_{0}\right),$$
where $C$ is a constant.
Since $\vert\vert\underline{\Lambda}^{\frac{1}{2}}v\vert\vert_{0}=\vert\vert v(x_1,0,t)\vert\vert_{\frac{1}{2}}$ it follows that $v\in C([0,T],H_{\frac{1}{2}}(\Gamma_1))$.\\
Using again Corollary $\ref{represformforv}$ we have that
$$v_t(x_1,0,t)=\int_{0}^{t}\left(\cos (t-s)\underline{\Lambda} ^{\frac{1}{2}}\right)P_{t}(x_1,0;s)\,ds+\underline{\Lambda}^{\frac{1}{2}}\sin (t\underline{\Lambda}^{\frac{1}{2}})v_0+\cos (t\underline{\Lambda}^{\frac{1}{2}})v_1,$$
from which we obtain that
$$\max_{0\leq t\leq T}\vert\vert v_{t}\vert\vert_{0}\leq \tilde{C}\left(\int_{0}^{T}\vert\vert P_{t}\vert\vert\,ds+\vert\vert v_{0}\vert\vert_{\frac{1}{2}}+\vert\vert v_{1}\vert\vert_{0}\right),$$
where $\tilde{C}$ is a constant. This shows that $v_t(x_1,0,t)\in C([0,T],L_{2}(\Gamma_1))$.
The assertions about $\eta$ follow from the conditions on $\Gamma_1$ in $\eqref{finalbvp}$.
\end{proof}
\section*{Acknowledgment}
The author would like to thank prof. Gregory Eskin, for pointing him
out this problem and for the numerous and inspiring conversations
during the writing of this paper.

\end{document}